\documentclass{article}
\usepackage[pdftex]{hyperref}
\hypersetup{
  colorlinks   = true, 
  urlcolor     = gray, 
  linkcolor    = red, 
  citecolor   =  blue
}
\usepackage{amsmath,amsfonts,amsthm,amssymb,graphicx,xcolor}
\usepackage[all,2cell,ps]{xy}

\theoremstyle{plain}
\newtheorem{theorem}[equation]{Theorem}

\newtheorem{lemma}[equation]{Lemma}

\theoremstyle{definition}

\newtheorem{remark}[equation]{Remark}

\newtheorem{question}[equation]{Question}

\newcommand{\IR}{\mathbb{R}}

\title{On the Boundary Injectivity Radius of Buser-Colbois-Dodziuk-Margulis Tubes}
\author{\small{Luca F. Di Cerbo} \\ \scriptsize{University of Florida}\\ \footnotesize{\textsf{ldicerbo@ufl.edu}}}
\date{}

\begin{document}

\maketitle

\begin{abstract}
We give a lower bound on the boundary injectivity radius of the Margulis tubes with smooth boundary constructed by Buser, Colbois, and Dodziuk. This estimate depends on the dimension and a curvature bound only. 
\end{abstract}

\vspace{10cm}

\tableofcontents

\vspace{1cm}


\section{Introduction}

Let $(M, g)$ be a compact manifold of dimension $n\geq 3$, with sectional curvatures satisfying
\[
-a^2\leq \sec_g\leq -1,
\]
for some constant $a\geq 1$. Given $x\in (M, g)$, we denote by $i_g(x)=i(x)$ its injectivity radius. The so-called \emph{thick-thin decomposition} of $(M, g)$ (\emph{cf.} Section 8 in the book \cite{BGS85}) ensures us there is a positive constant
\begin{align}\label{muconst}
\mu=a^{-1}c_{n},
\end{align}
$c_n>0$ depending on the dimension only, so that if the set
\begin{align}\label{muset}
M_{\mu}:=\{x\in M\quad | \quad i(x)<\mu\}
\end{align}
is not empty, it is then the union of a \emph{finite} number of disjoint tubes $\{T_{\zeta}\}$ around short closed geodesics $\{\zeta\}$. Each $T_\zeta$ is homeomorphic to $\zeta\times B^{n-1}$, with $B^{n-1}$ a closed ball in $\IR^{n-1}$. Unfortunately, the boundary of $T_{\zeta}$ is \emph{not} usually smooth! Nevertheless, the fact that one understands the topology of regions with small injectivity radius is extremely important in the study of the geometry and topology of quotients of Hadamard manifolds. For example, the study of normalized Betti numbers presented in \cite{ABBG18} heavily relies on this \emph{coarse} thick-thin decomposition.

From a geometric analysis point of view, the lack of smoothness of the tubes $\{T_{\zeta}\}$ is a somewhat undesirable feature of the usual (or coarse) thick-thin decomposition. In many questions arising from the spectral geometry of compact quotients of Hadamard manifolds, it is indeed useful to work with a \emph{smooth} thick-thin decomposition due to Buser, Colbois and Dodziuk \cite{BCD93}. In this decomposition, one selects a subset $\{\gamma\}\subset\{\zeta\}$ of closed geodesics shorter than an explicit bound depending on $\mu$, and tubes $V_{\gamma}\subset T_{\gamma}$ with \emph{smooth} boundaries. For the details, we refer to the statement of Theorem \ref{mainBCD} below (\emph{cf}. Theorem 2.14 in \cite{BCD93}). See also the recent paper of Hamenst\"adt \cite{Ham18} for more on this circle of ideas.

Recently, M. Stern and the author also employed the smooth think-thin decomposition of Buser, Colbois and Dodziuk to study normalized Betti numbers of negatively curved manifolds, see \cite{DS19}. In this study, we are able to extend some of the convergence results proved in \cite{ABBG18}, and in some cases recover some of their results. Our approach heavily relies on the \emph{smooth} thick-thin decomposition.

The goal of this paper is to derive a boundary injectivity radius estimate for the smooth thick-thin decomposition of Buser-Colbois-Dodziuk. This estimates depends on the dimension and a curvature bound only, and it seems particularly useful when studying the manifold with boundary obtained from $(M, g)$ by removing the tubes $\{V_{\gamma}\}$ with smooth boundaries constructed in Theorem \ref{mainBCD} (\emph{cf}. Theorem 2.14 in \cite{BCD93}). Concluding, we believe this estimate to be a desirable feature to be added to the smooth thin-thick decomposition, and it should make it more readily available for a multitude of geometric analysis questions regarding compact quotients of Hadamard manifolds.\\

We conclude this introduction with the statement of our main theorem.

\begin{theorem}\label{improved}
	Let $(M, g)$ be a compact Riemannian manifold of dimension $n\geq 3$, with sectional curvature 
	\[
	-a^2\leq\sec_g\leq -1,
	\]
	for some constant $a\geq 1$. Let $(N, g)$ be the manifold with boundary obtained by removing the smooth disjoint tubes $\{V_{\gamma}\}$ constructed in Theorem \ref{mainBCD} (\emph{cf}. Theorem 2.14 in \cite{BCD93}) from $(M, g)$, and for each $\gamma$ set
	\[
	H_{\gamma}=\partial V_{\gamma}.
	\]
	There exists a constant $T(a, n)>0$ depending on $n$ and $a$ only such that the exponential map restricted to the normal bundle of $\partial N$
	\[
	\exp^{\perp}: \partial N\times [0, T(a, n))=\bigcup_{\gamma}H_{\gamma}\times [0, T(a, n))\longrightarrow M
	\]
	is a diffeomorphism onto its image.
\end{theorem}

For the details of the proof, we refer to Section \ref{conclusion}.

\vspace{0.1 in}
\noindent\textbf{Acknowledgments}. The author would like to thank Professor Mark Stern for asking the question addressed in this paper, and for several early discussions regarding this matter. Finally, he thanks the two referees for useful and constructive comments.

\section{Buser-Colbois-Dodziuk Thick-Thin Decomposition}\label{newTT}

In this section, we recall some features of the smooth thick-thin decomposition for compact quotients of Hadamard manifolds construction given by Buser, Colbois, and Dodziuk (\emph{cf}. Section 2 in \cite{BCD93}). We also fix the notation and, in Lemma \ref{lbound}, we summarize a few results we derived in \cite{DS19} concerning this construction. 

Let $(M, g)$ be a compact manifold of dimension $n\geq 3$, with sectional curvatures satisfying
\[
-a^2\leq \sec_g\leq -1,
\]
for some constant $a\geq 1$. We set
\[
\mu=a^{-1}c_{n},
\]
with $c_n>0$ depending on the dimension only, and
\[
M_{\mu}:=\{x\in M\quad | \quad i(x)<\mu\}
\]
as in Equations \eqref{muconst} and \eqref{muset}.  $M_{\mu}$ is the union (possibly empty) of disjoint tubes $\{T_{\gamma_i}\}$ around short closed geodesics $\{\gamma_i\}$. For simplicity, we require $\mu\leq 1$. For every tube $T_{\gamma}$, the core geodesic $\gamma$ has length $\ell(\gamma)<2\mu$. For every point $p\in \gamma$, and every tangent vector $v\in T_{p}M$ perpendicular to $\gamma'(0)$, let $\delta_{p, v}(t)$ denote the unit speed geodesic ray emanating from $p$ in the direction of $v$. We call these rays {\em radial arcs} and their tangent vector fields the radial vector field $\mathcal{R}$. 

In every interval $[0, t_0]$ such that $i(\delta_{p, v}(t))\leq \mu$ for $t\in [0, t_0]$,  the function $t\rightarrow i(\delta_{p, v}(t))$ is strictly monotonic increasing. Thus, there exists $R_{p, v}>0$ depending on the initial condition of the geodesic ray such that $i(\delta(R_{p, v}))=\mu$ and $i(\delta_{p, v}(t))<\mu$ for any $t\in[0, R_{p, v})$. The arc $\delta([0, R_{p, v}])$ is called the \emph{maximal arc}. Also, different radial arcs are disjoint with the exception possibly of their initial points. As we mentioned in the introduction, different maximal radial arcs in $T_{\gamma}$ may have very different lengths and the boundary of $T_{\gamma}$ needs not be smooth in general. We refer to Section 8 in the book \cite{BGS85} for more background material on the construction of the coarse thick-thin decomposition. 

Next, we need to recall Equation 2.5 in \cite{BCD93}, see also Lemma 2.4 therein.

\begin{lemma}[Equation 2.5 in \cite{BCD93}]\label{BCD1}
	There exist constants $C_1$, $C_2$ depending only on the dimension $n$, such that if
	\begin{align}\label{newconst}
	\ell(\gamma)\leq C_1\exp(-C_2 a)\mu^n a^{n-1},
	\end{align}
	then $d(x, \gamma)\geq 10$ for every $x\in T_{\gamma}$ with $i(x)=\mu/2$.
\end{lemma}

We now declare a geodesic to be \emph{small} if and only if its length satisfies the bound \eqref{newconst} of Lemma \ref{BCD1}. This means we may not consider many small geodesics in the usual thick-thin decomposition of $M$. In other words, we only consider small geodesics which posses large Margulis tubes around them. This fact plays a role in the results that follow.

Next, given a geodesic $\gamma$ satisfying \eqref{newconst}, we look at the following tube around it:
\begin{align}\label{newtube}
U_{\gamma}:=\{x\in T_{\gamma}\quad | \quad i(x)\leq\mu/2\}.
\end{align}
Again, there is no a priori reason to believe that the boundary $\partial U_{\gamma}$ is smooth. Nevertheless, the following theorem of Buser, Colbois and Dodziuk ensures the existence of a small deformation of $U_{\gamma}$ with smooth boundary. 

\begin{theorem}[Theorem 2.14 in \cite{BCD93}]\label{mainBCD}
	Let $\gamma$ be a geodesic in $M$ satisfying \eqref{newconst}. There exists a smooth hypersurface $H_{\gamma}$ contained in $T_{\gamma}\setminus\gamma$ with the following properties:
	\begin{itemize}
		\item The angle $\theta$ between the radial vector field $\mathcal{R}$ and the exterior normal of $H_{\gamma}$ is less that $\pi/2-\alpha$ for a constant $\alpha=\alpha(a, n)\in (0, \pi/2)$. 
		\item The sectional curvatures of $H_{\gamma}$ with respect to the induced metric are bounded in absolute value by a constant depending only on $a$ and $n$.
		\item $H_{\gamma}$ is homeomorphic to $\partial U_{\gamma}$ by pushing along radial arcs. The distance between $x\in H_{\gamma}$ and its image $\bar{x}\in\partial U_{\gamma}$ satisfies $d(x, \bar{x})\leq\mu/50$.
	\end{itemize}
\end{theorem}  

Given a geodesic $\gamma$ satisfying \eqref{newconst}, we consider the tube $V_\gamma$ around it defined by the requirements:
\begin{align}\label{finaltube}
H_{\gamma}=\partial V_{\gamma},\text{ and }\gamma\in V_\gamma.
\end{align}
These new tubes always have \emph{smooth} boundaries, and the union 
\[
\{V_\gamma\},\quad M\setminus \bigcup_\gamma V_{\gamma},
\]
gives the smooth thin-thick decompostion. We close this section with an estimate on the injectivity radius of points on the smooth boundaries $\{H_{\gamma}\}$, and with an estimate on the distance between the tubes $\{V_{\gamma}\}$. The following lemma is a collection of results proved in \cite{DS19}. It will be used at the end of the main argument in Section \ref{conclusion}.

\begin{lemma}[Lemma 34 and Lemma 36 in \cite{DS19}]\label{lbound}
	Let $\gamma$ be a geodesic in $M$ satisfying \eqref{newconst}. For any point $x\in H_{\gamma}$, we have
	\[
	\frac{26}{50}\mu\geq i(x)\geq \frac{24}{50}\mu.
	\] 
	Moreover, if $\gamma\neq \zeta$ are two distinct closed geodesics in $M$ satisfying \eqref{newconst}, then
	\[
	d(H_{\gamma}, H_{\zeta})>\frac{48}{50}\mu.
	\]
\end{lemma}

\begin{remark}
	In Section 4.2 of \cite{DS19}, the interested reader can find further effective bounds for the smooth thick-thin decomposition of Buser-Colbois-Dodziuk. In particular, Lemma 37 and Lemma 41 in \cite{DS19} provide refined bounds on the volume and length of the smallest geodesic arc for any smooth tube around a short geodesic. 
\end{remark}

\section{Boundary Injectivity Radius}\label{conclusion}

In this section, all closed geodesics are assumed to be small, that is their lengths satisfy the bound in Equation \eqref{newconst} of Lemma \ref{BCD1}. Given a small closed geodesic $\gamma$, we now study the normal bundle of the hypersurface $H_{\gamma}\subset M$.  Let $NH_{\gamma}$ be the normal bundle of $H_{\gamma}$ in $(M, g)$, i.e., the set of all tangent vectors to $M$ at points in $H_{\gamma}$ that are perpendicular to $H_{\gamma}$. Let 
\[
\exp^{\perp}: NH_{\gamma}\longrightarrow M
\]
be the usual Riemannian exponential map restricted to $NH_{\gamma}$. Since $H_{\gamma}$ is compact, for $\epsilon>0$ small enough if we restrict $\exp^{\perp}$ to tangent vectors $w$ with $g(w, w)\leq \epsilon$, then $\exp^{\perp}$ is a diffeomorphism onto its image. Here we would like to find a $t>0$ independent of the particular hypersurface $H_{\gamma}$ such that $\exp^{\perp}$ is a diffeomorphism on $H_{\gamma}\times [0, t)$. We call the \emph{boundary injectivity radius}, say $i_{H_{\gamma}}$, the largest such $t$. Here, we denote by $H_{\gamma}\times [0, t)$ a one sided open neighborhood of the zero section in $NH_{\gamma}$. Moreover, we use positive value of $t$ to indicate those geodesic that start on $H_{\gamma}$ and points away from the core geodesic $\gamma$.\\

We now seek a uniform lower bound of the boundary injectivity radius of such tubes. First, we need to avoid critical values of $\exp^{\perp}$. It is well known that $x\in M$ is a critical value of $\exp^{\perp}$ if and only if $x$ is a \emph{focal} point of $H_{\gamma}\subset M$, see Proposition 4.4 page 231 in the book \cite{doC92}. Now the notion of focal point can be given in terms of Jacobi fields along geodesic, see Section 4 in Chapter 10 of the book \cite{doC92}. More precisely, given the submanifold $H_{\gamma}\subset M$, we call $q\in M$ a focal point of $H_{\gamma}$ if there exists a geodesic $\nu: [0, \l]\rightarrow M$, with $\nu(0)\in H_{\gamma}$, $\nu^\prime(0)\in (T_{\nu(0)}H_{\gamma})^\perp$, $\nu(\l)=q$, and a non-zero Jacobi vector field $J$ along $\nu$ with $J(\l)=0$ and such that:
\[
J(0)\in T_{\nu(0)}H_{\gamma}, \quad J^\prime(0)+S_{\nu^\prime(0)}(J(0))\in (T_{\nu(0)}H_{\gamma})^\perp,
\]
where $S_{\nu(0)}$ is the linear self-adjoint operator on $T_{\nu(0)}H_{\gamma}$ associated to the second fundamental form of $H_{\gamma}$. Note that these requirements imply that 
\begin{align}\label{Jort}
J^\prime(0)+S_{\nu^\prime(0)}(J(0))=0 \Rightarrow J^\prime(0)\in T_{\nu(0)}H_{\gamma}.
\end{align}
Indeed, as $H_{\gamma}\subset M$ is a hypersurface we have the isometric splitting
\[
T_{\nu(0)}M=T_{\nu(0)}H_{\gamma}\oplus\IR\nu^{\prime}(0),
\]
so that if by contradiction we assume Equation \eqref{Jort} not to be true, we would have:
\[
J(t)=kt\nu^\prime(t)+\bar{J}(t),
\]
where $k\in \IR$ is non-zero and where $\bar{J}(t)$ is a Jacobi field along $\nu$ with
\[
\bar{J}(0)=J(0), \quad \bar{J}^\prime(0)\in T_{\nu(0)}H_{\gamma},
\]
so that $\bar{J}(t)\perp\nu(t)$ for all $t$. This implies that $J(\l)\neq 0$ which is then a contradiction.

We can then claim that $J^\prime(0)\in T_{\nu(0)}H_{\gamma}$ which combined with the fact that $J(0)\in T_{\nu(0)}H_{\gamma}$ implies
\[
J(t)\perp \nu(t), \quad t\in[0, \l].
\]
By linearity of Jacobi's equation, we next decompose 
\[
J(t)=J_{1}(t)+J_{2}(t),
\]
where $J_{1}$ and $J_{2}$ are Jacobi vector fields satisfying the initial conditions:
\[
J_{1}(0)=J(0),\quad J^\prime_{1}(0)=0, \quad \quad J_{2}(0)=0,\quad  J^\prime_{2}(0)=J^\prime(0).
\]
Next, normalize $J$ so that $|J(0)|=1$ and assume that we have
\[
J^\prime(0)=-S_{\nu^\prime(0)}(J(0))=-\lambda J(0),
\]
where $\lambda$ is the largest eigenvalue of $S_{\nu^\prime(0)}$. In other words, $\lambda$ is the \emph{largest} of the principal curvatures of $H_{\gamma}$ at the point $\nu(0)$. Assume also for a moment that the sectional curvature of $(M, g)$ satisfies $\sec_g=K$, with $K$ a negative constant. If $W(t)$ is the parallel transport of $J(0)$ along $\nu(t)$, a simple computation gives us the following:
\[
J_{1}(t)=\cosh{t\sqrt{-K}}W(t), \quad J_{2}(t)=\frac{\sinh{t\sqrt{-K}}}{\sqrt{-K}}(-\lambda)W(t).
\] 
Thus, in a space form of negative sectional curvature $K$ if the largest principal curvature of a hypersurface $H_{\gamma}$ is $\lambda$, i.e., $S\leq \lambda I$, we then have that the closest focal point, say $q$, satisfies
\[
d(H_{\gamma}, q)=t
\]
where $t$ is a solution of: 
\[
\coth{t\sqrt{-K}}=\frac{\lambda}{\sqrt{-K}}.
\]
We can now state a general lemma concerning the focal set of $H_{\gamma}$ in M.

\begin{lemma}\label{focal}
	Let $(M, g)$ be a compact manifold of dimension $n\geq 3$, with sectional curvature 
	\[
	-a^2\leq\sec_g\leq-1,
	\]
	for some constant $a\geq 1$. Let $H_{\gamma}$ be the smooth boundary of any of the Buser-Colbois-Dodziuk tubes in $M$. For any focal point $q$ of $H_{\gamma}$ we have
	\[
	d(H_{\gamma}, q)\geq t,
	\] 
	where $t$ is a solution of:
	\[
	\coth{t}=\lambda
	\]
	with $\lambda$ a constant depending on $a$ and $n$ only.
\end{lemma}

\begin{proof}
	Because of Theorem \ref{mainBCD} (see the bound on the second fundamental form at the end of page 12 in proof of Theorem 2.14 in \cite{BCD93}), given any $H_{\gamma}\subset M$ and any $p\in H_{\gamma}$, we have that the linear self-adjoint operator associated to the second fundamental form of $H_{\gamma}$ at $p$
	\[
	S_p: T_p H_{\gamma}\longrightarrow T_p H_{\gamma}
	\]
	satisfies $S_p\leq \lambda(a, n) I_p$ where $\lambda$ is a constant depending on $a$ and $n$ only and $I_p: T_p H_{\gamma}\rightarrow T_p H_{\gamma}$ is the identity. We can now use Rauch's comparison theorem for submanifolds (\emph{cf}. Theorem 4.9 page 234 in \cite{doC92}) in conjunction with the Jacobi vector field argument given above and conclude the proof. In this case, the model manifold has constant negative sectional curvature equal to $-1$. 
\end{proof}

Thanks to Lemma \ref{focal}, we conclude that the focal set is bounded away from any of the $H_{\gamma}$ uniformly in terms the dimension and a bound on the curvature.
We now focus on the next obstruction for $\exp^{\perp}$ to be a diffeomorphism onto its image, namely the existence of certain short geodesics. More precisely, we have the following lemma.  
\begin{lemma}\label{short}
	Let $l>0$ be the smallest positive number such that
	\[
	\exp^{\perp}: H_{\gamma}\times [0, l)\longrightarrow M
	\]
	is a diffeomorphism. Let $\nu: [0, l)\rightarrow M$ be a normalized geodesic with $\nu^\prime(0)\in (T_{\nu(0)}H_{\gamma})^\perp$ such that
	\begin{align}\label{min}
	d(H_{\gamma}, \nu(t))\neq t
	\end{align}
	for $t>l$. If $\nu(l)$ is not a focal point of $H_{\gamma}$, we have that $\nu(2l)\in H_{\gamma}$ and $\nu^\prime(2l)\in (T_{\nu(2l)}H_{\gamma})^\perp$.
\end{lemma}

\begin{proof}
	
	Because of \eqref{min}, there exists a sequence of real numbers $\{\epsilon_i\}$ with
	\[
	\epsilon_i>0, \quad \lim_{i\to\infty}\epsilon_i=0,
	\]
	 such that
	\begin{align}\label{def}
	\nu(l+\epsilon_i)=\sigma_i(l+\epsilon^\prime_i)
	\end{align}
	and
	\[
	-\epsilon_i\leq \epsilon^\prime_i \leq\epsilon_i,
	\]
	where $\sigma_{i}$ is a normalized minimizing geodesic from $H_{\gamma}$ to $\nu(l+\epsilon_i)$ with $\sigma^\prime_{i}(0)\in (T_{\sigma_{i}(0)}H_{\gamma})^\perp$. Thus, $(\sigma_{i}(0), \sigma^\prime_{i}(0))$ is a sequence in the unit tangent normal bundle of $H_{\gamma}$ in $M$, say $N^{1}H_{\gamma}$. By compactness of $N^1H_{\gamma}$ we have up to a subsequence:
	\[
	\lim_{i\to\infty}(\sigma_{i}(0), \sigma^\prime_{i}(0))=(q, v)
	\] 
	where $q\in H_{\gamma}$ and $v\in (T_{q}H_{\gamma})^\perp$ with $|v|=1$. Note that by continuity, the limiting geodesic say $\sigma:[0, l]\rightarrow M$ with initial conditions
	\[
	(q, v)=(\sigma(0), \sigma^\prime(0))
	\]
	has to be minimizing. We now claim that
	\[
	(q, v)\neq (\nu(0), \nu^\prime(0)),
	\]
	so that the geodesics $\nu$ and $\sigma$ have to be distinct. By contradiction, assume that this is not the case. Let $U$ be an open neighborhood of $\nu(l)\in M$. Let $V$ be an open neighborhood of $(\nu(0), l\nu^\prime(0))\in NH_{\gamma}$ such that
	\[
	\exp^\perp: V\longrightarrow U
	\]
	is a diffeomorphism. For $i$ large enough, we have $(\sigma_{i}(0), (l+\epsilon^\prime_{i})\sigma^\prime_{i}(0))\in V$. Because of \eqref{def}, we have
	\[
	\exp^{\perp}(\nu(0), (l+\epsilon_{i})\nu^\prime(0))=\exp^{\perp}(\sigma_i(0), (l+\epsilon^\prime_{i})\sigma_i^\prime(0)),
	\]
	which then implies that
	\[
	\sigma_i(0)=\nu(0), \quad \epsilon_i=\epsilon^\prime_i, \quad \nu^\prime(0)=\sigma^\prime_i(0).
	\]
	Thus, $\nu$ is minimizing up to time $l+\epsilon_i$, and this is a contradiction. It remains now to prove that 
	\begin{align}\label{def2}
	\nu^\prime(l)=-\sigma^\prime(l).
	\end{align}
	If this is not the case, there exists $v\in T_{\nu(l)}M$ such that
	\[
	\langle v, \nu^\prime(l) \rangle<0, \quad \langle v, \sigma^\prime(l) \rangle<0.
	\]
	Let $\tau: [-\delta, \delta]\rightarrow M$ be a curve with $\tau(0)=\nu(l)=\sigma(l)$ and $\tau^\prime(0)=v$. Let $V, W$ be neighborhoods of $(\nu(0), l\nu^\prime(0))$, $(\sigma(0), l\sigma^\prime(0))$ respectively such that
	\[
	\exp^\perp: V\longrightarrow U, \quad \exp^\perp: W\longrightarrow U
	\] 
	are diffeomorphisms, and where $U$ is a neighborhood of $\nu(l)$ containing $\tau[-\delta, \delta]$. Let 
	\[
	v_1: [-\delta, \delta]\rightarrow V, \quad v_2: [-\delta, \delta]\rightarrow W
	\]
	 be curves such that:
	\[
	\exp^\perp(v_1(s))=\tau(s), \quad \exp^\perp(v_2(s))=\tau(s).
	\]
	Thus, we have
	\[
	v_1(s)=(\alpha_1(s), \beta_1(s)), \quad v_2(s)=(\alpha_2(s), \beta_2(s)),
	\]
	where $\alpha_1$, $\alpha_2$ are smooth curves in $H_{\gamma}$ and where for any $s\in [-\delta, \delta]$:
	\[
	\beta_{1}(s)\in (T_{\alpha_1(s)}H_{\gamma})^\perp, \quad \beta_{2}(s)\in (T_{\alpha_2(s)}H_{\gamma})^\perp.
	\]
	Define two variations for $s\in[-\delta, \delta]$, $t\in [0, l]$ as follows
	\[
	\gamma^1_s(t)=\exp^\perp\Big(\alpha_1(s), \frac{t}{l}\beta_{1}(s)\Big), \quad \gamma^2_s(t)=\exp^\perp\Big(\alpha_2(s), \frac{t}{l}\beta_{2}(s)\Big),
	\]
	and notice that
	\[
	\gamma^1_s(0)=\alpha_1(s), \quad\gamma^2_s(0)=\alpha_2(s), \quad\gamma^1_{s}(l)=\gamma^2_{s}(l), \quad \gamma^1_{0}(t)=\nu(t), \quad \gamma^2_{0}(t)=\sigma(t).
	\]
	Now a simple application of the first variation of length (or energy) along these variations gives that for $s>0$ sufficiently small:
	\[
	\ell(\gamma^1_s)<\ell(\nu), \quad \ell(\gamma^2_s)<\ell(\sigma).
	\]
	Since $\gamma^1_s(l)=\gamma^2_s(l)$ we get a contradiction. The proof is complete.
\end{proof}

Given $H_{\gamma}\subset M$, we can derive an effective lower bound for the length of short geodesics as in Lemma \ref{short}. This estimate is expressed in terms of the Margulis constant $\mu$ of $(M, g)$.

\begin{lemma}\label{boundaryinjectivity}
	Let $\nu: [0, 2l]\rightarrow M$ be a normalized geodesic in $M$  with $\nu^\prime(0)\in (T_{\nu(0)}H_{\gamma})^\perp$ and $\nu^\prime(2l)\in (T_{\nu(2l)}H_{\gamma})^\perp$ of minimal length. We then have
	\[
	\frac{1}{2}\ell(\nu)=l\geq \frac{\mu}{200},
	\]
	where $\mu$ is the Margulis constant.
\end{lemma}

\begin{proof}
    Suppose this is not the case. There exist $p, q\in H_{\gamma}$ and a normalized geodesic $\nu: [0, 2l]\rightarrow M$ with
	\[
	\nu(0)=p, \quad \nu(2l)=q, \quad \nu^\prime(0)\in (T_{p}H_{\gamma})^\perp, \quad \nu^\prime(2l)\in (T_{q}H_{\gamma})^\perp,
	\]
	such that $2l=d(p, q)<\mu/100$. By Lemma \ref{lbound},  we have $i(p)\geq24/50\mu$ so that $\nu$ is indeed length minimizing for any $t\in[0, 2l]$. Next, we denote by $\delta^x$ and $\delta^y$ the distinct normalized radial arcs emanating from $x, y \in \gamma$ and passing through $q$ and $p$ respectively. Now given $p \in \delta^y$ as before, let
	\[
	L=d(p, \delta^x)=d(p, x^\prime),
	\]
	where $x^\prime$ is a point on $\delta^x$. To show the existence of such $x^\prime$ argue as follows. First, note that the arc $\delta^x$ extends  past the point $q$ for at least a segment of length $\frac{24}{50}\mu$. Indeed, recall that the radial arc $\delta^x$ stops when it reaches the first point $z\in M$ with $i(z)=\mu$. By Lemma \ref{lbound}, we then have:
	\begin{align}\label{q}
	d(q, z)\geq |i(z)-i(q)|\geq \frac{24}{50}\mu,
	\end{align}
	which implies
	\[
	d(p, z)\geq |d(p, q)-d(q, z)|\geq \frac{24}{50}\mu-\frac{\mu}{100}=\frac{47}{100}\mu.
	\]
	Also, the point $q\in \delta^x$ is quite far from the core geodesic $\gamma$. Indeed, by combining Lemma \ref{BCD1} and Theorem \ref{mainBCD}, we have
	\begin{align}\label{distancecore}
	d(q, \gamma)\geq 10-\frac{\mu}{50},
	\end{align}
	which then implies
	\[
	d(p, x)\geq 9.
	\]
	Thus, since we assumed $\ell(\nu)<\mu/100$ and $\nu$ connects $p$ with $q\in\delta^x$, we have the existence of $x^\prime\in\delta^x$ realizing the distance $L=d(p, \delta^x)<\mu/100$. We denote by $\beta: [0, L]\rightarrow M$ the normalized geodesic with $\beta(0)=p$, $\beta(L)=x^\prime$, realizing such a distance. Observe that $\beta$ is perpendicular to $\delta^x$ at the intersection point $x^\prime$. One may wonder if we can have $x^\prime=q$ and $\beta=\nu$. This is not case. To show this, let $\alpha_{p}$ the angle between $\delta^y$ and $\nu$ at $p$, and $\alpha_{q}$ the angle between the $\delta^x$ and $-\nu$ at $q$. If $x^\prime=q$, we would have $\alpha_{q}=\pi/2$. On the other hand, by Theorem \ref{mainBCD} we know that $0\leq \alpha_q< \pi/2-\alpha(n, a)$, where $\alpha(n, a)\in (0, \pi/2)$. We could still have $x^\prime=q$ and $\beta\neq\nu$. This is not the case either. In fact, if this would imply the existence of two distinct geodesics $\nu$ and $\beta$ starting at $p$ and meeting at $q$ with length less than $\mu/100$. This implies $i(p)<\mu/100$ which contradicts Lemma \ref{lbound}. 
	
	Now that we know that $\beta\neq\nu$ and $x^\prime\neq q$, the idea is to look at the geodesic triangle whose vertices are $x^{\prime}$, $p$ and $q$. Denote by $l_1(=\beta)$ the segment joining $p$ and $x^\prime$ and by $l_2$ the segment between $x^\prime$ and $q$. Next, we denote by $\gamma_1$ the angle between $l_2$ and $\nu$ at $q$, by $\gamma_2$ the angle between $\nu$ and $l_{1}$ at $p$, and finally by $\gamma_3=\pi/2$ the angle between $l_1$ and $l_2$. By the triangle inequality notice that
	\begin{align}\label{distance}
	d(x^\prime, q)<\frac{\mu}{100}+\frac{\mu}{100}=\frac{\mu}{50},
	\end{align}
	so that we have a well defined geodesic triangle between the vertices $x^\prime$, $p$ and $q$ and the segment $l_{2}$ is just a piece of the arc $\delta^x$. Now, the this geodesic triangle is within a manifold with $-a^2\leq\sec_{g}\leq -1$ so that:
	\[
	\gamma_{1}+\gamma_{2}+\gamma_{3}<\pi,
	\]
	see for example Lemma 3.1 in \cite{doC92}.
	On the other hand:
	\[
	\gamma_{1}+\gamma_{2}+\gamma_{3}= \pi-\alpha_q+\gamma_2+\frac{\pi}{2}> \frac{3}{2}\pi-\alpha_q.
	\]
	But now, because of Theorem \ref{mainBCD} we have
	\[
	\frac{3}{2}\pi-\alpha_q\geq \pi+\alpha(a, n)>0,
	\]
	so that we obtain the contradiction
	\[
	\pi>\gamma_{1}+\gamma_{2}+\gamma_{3}\geq \pi+\alpha(n, a)>\pi.
	\]
	Finally, notice that the degenerate case $\alpha_q=0$ cannot occur. Indeed, if this was the case we would have that the geodesic $\nu$ coincides with the piece of the radial arc $\delta^x$ joining $q$ with $z$. By \eqref{q} we know that $d(q, z)\geq 24/50\mu$, so that we cannot possibly have $l(\nu)<\mu/100$. Similarly for the case $\alpha_p=0$. The proof is complete.
\end{proof}

We can now prove the main theorem stated in the introduction.

\begin{proof}[Proof of Theorem \ref{improved}]
	Given $(N, g)$, consider the exponential map restricted to the normal bundle of $\partial N$
	\[
	\exp^\perp: \partial N\times [0, \infty)\longrightarrow M.
	\]
	Now, the boundary $\partial N$ is the disjoint union of finitely many smooth hypersurfaces $\{H_{\gamma}\}$. By Lemma \ref{lbound}, these hypersufaces satisfy
	\[
	d(H_{\gamma}, H_{\gamma^\prime})>\frac{48}{50} \mu,
	\] 
	for $\gamma\neq\gamma^\prime$. Lemma \ref{focal} and Lemma \ref{short} now provide estimates on the focal set and short geodesics for any of the $H_{\gamma}$'s in $\partial N$. These estimates depend upon the dimension of $M$ and a lower bound on the sectional curvature only. As the Margulis constant $\mu$ can be estimated in terms of these constants as well, the proof is complete. 
\end{proof}



\begin{thebibliography}{ELMNPM}

\bibitem[ABBG18]{ABBG18} M. Abert, N. Bergeron, I. Biringer, T. Gelander, Convergence of normalize Betti numbers in nonpositive curvature. \textit{arXiv:1811.02520v2} [math.GT].

\bibitem[BGS85]{BGS85} W. Ballmann, M. Gromov, V. Schroeder, Manifolds of nonpositive curvature. Progress in Mathematics, 61. \textit{Birkh\"auser Boston, Inc., Boston, MA,} 1985 

\bibitem[BCD93]{BCD93} P. Buser, B. Colbois, J. Dodziuk, Tubes and Eigenvalues for Negatively Curves Manifolds. \textit{J. Geom. Anal.} \textbf{3} (1993), no. 1, 1-26.

\bibitem[DS19]{DS19} L. F. Di Cerbo, M. Stern, Harmonic Forms, Price Inequalities, and Benjamini-Schramm Convergence. \textit{arXiv:1909.05634v1} [math.DG].

\bibitem[doC92]{doC92} M. P. do Carmo, Riemannian Geometry. Mathematics: Theory \& Applications. \textit{ Birkhäuser Boston, Inc., Boston, MA,} 1992.

\bibitem[Ham19]{Ham18} U. Hamenst\"adt, Small eigenvalues in thick-thin decomposition in negative curvature. \textit{Ann. Inst. Fourier (Grenoble)} \textbf{69} (2019), no. 7, 3065-3093.

\end{thebibliography}
\end{document}